\documentclass[12pt,reqno]{amsart}
\usepackage[english]{babel}  \usepackage{latexsym,amsfonts} \usepackage{amsmath,amssymb,amsthm} \usepackage{mathrsfs}
\usepackage{enumerate} \usepackage{lmodern} \usepackage{mathpazo}   \usepackage{euscript}  \usepackage[pdftex]{hyperref}
\numberwithin{equation}{section}  \makeatletter\@addtoreset{equation}{section}
\newtheorem {theorem}{Theorem}[section]        
   \newtheorem {corollary}[theorem]{Corollary}     \newtheorem {remark}[theorem]{Remark}
\newtheorem {proposition}[theorem]{Proposition}       
      
\newcommand{\C}{\mathbb C}       \newcommand{\Z}{\mathbb Z} 	 

    \newcommand{\bz}{\overline{z}}  
   
\newcommand{\minmn}{m\wedge n}  \newcommand{\minms}{m\wedge s}
 \newcommand{\minnr}{n\wedge r}
\begin{document}

\title[]{Generalized Zernike polynomials: Operational formulae and generating functions}
\thanks{A. G. is partially supported by the Hassan II Academy of Sciences and Technology}


    \author[]
    {Bouchra Aharmim}
    \address[B. Aharmim]{Department of Mathematics and Informatics,  Faculty of Sciences, Ben M'sik, Hassan II University, Casablanca, Morocco}
     \email{bouchra.aharmim@gmail.com}
    \author[]
    {Amal El Hamyani}
    \address[A. El Hamyani]{Laboratory of Analysis and Applications-URAC/03, Department of Mathematics, P.O. Box 1014,  Faculty of Sciences, Mohammed V University, Rabat, Morocco}
     \email{amalelhamyani@gmail.com}
    \author[]
    {Fouzia El Wassouli}
    \address[F. El Wassouli]{Department of Mathematics and Informatics,  Faculty of Sciences, A\"{i}n Chock, Hassan II University, Casablanca, Morocco}
    \email{elwassouli@gmail.com}
   \author[]
   {Allal Ghanmi}
    \address[A. Ghanmi]{Laboratory of Analysis and Applications-URAC/03, Department of Mathematics, P.O. Box 1014,  Faculty of Sciences, Mohammed V University, Rabat, Morocco}
    \email{ag@fsr.ac.ma}

\begin{abstract}
We establish new operational formulae of Burchnall type for the complex disk polynomials (generalized Zernike polynomials). We then use them
           to derive some interesting identities involving these polynomials. In particular, we establish
            recurrence relations with respect to the argument and the integer indices, as well as Nielsen identities and Runge addition formula. In addition, various new generating functions for these disk polynomials are proved.
\end{abstract}
\keywords{
Complex disk polynomials; Operational formulae; Nielsen's identities; Runge's addition formula; Generating functions
}
\subjclass[2010]{33C45; 42C05; 58C40}
\maketitle

\section{Introduction} 

Operational formulae of Burchnall type are powerful tools in the theory of orthogonal polynomials \cite{Burchnall41,GouldHopper62,Chatterjea63a,Al-Salam64,Al-SalamIsmail75}.
They have been used to obtain new identities or to give simpler proofs of well-known ones.
In \cite{Burchnall41}, Burchnall employed the operational formula
 \begin{align*}
              \left(- \frac{d}{dx} +2x\right)^{m}(f) = m! \sum\limits_{k=0}^{m} \frac{(-1)^{k}}{k!}
              \frac{H_{m-k}(x)}{(m-k)!} \frac{d^k}{dx^k}(f).
          \end{align*}
 to prove, in a direct and simple way, the Nielsen identity \cite{Nielsen18}
as well as the Runge addition formula \cite{Runge14} for the classical Hermite polynomials $H_{m}(x)$.
Since then, many extensions to specific  polynomials in one real variable have been obtained
\cite{Chatterjea63a,PatilThakare78,ChatterjeaSrivastava93,PurohitRaina09,Kaanoglu12}.
For example, in \cite{Singh65}, R.P. Singh has developed  an operational relation
 for the Jacobi polynomials $ \displaystyle P^{(\alpha,\beta)}_n(x)$ which was used to derive some useful properties such as
  the quadratic recurrence formula
\begin{align*}
P^{(\alpha,\beta)}_{n+m}(x)   &:= \frac{n!m!}{(n+m)!}\sum_{k=0}^n 
 \frac{(-1)^k(\alpha+\beta+2n+m+1)_k}{2^{2k} k!}
 \\&\times
 (1-x^2)^{k}P^{(\alpha+k,\beta+k)}_{n-k}(x)P^{(\alpha+n+k,\beta+n+k)}_{m-k}(x).
\end{align*}
 In the present paper, we develop some new operational formulae for the complex analogue of the Jacobi polynomials, the
  disk polynomials in the two conjugate complex variables $z,\bar z$  in the complex unit disk $D\subset \C$.
  We define them here as
  \begin{align}\label{ZernikePol}
  \mathcal{Z}_{m,n}^\gamma (z,\bar z)
   = (-1)^{m+n}  (1- |z|^2)^{-\gamma} \frac{\partial^{m+n}}{\partial z^m \partial \bar z^n}\left( (1- |z|^2)^{\gamma+m+n} \right).
  \end{align}
Notice that the normalization adopted here differs from the one adopted for the disk polynomials considered by
W\"unsche \cite{Wunsch05} and denoted $ \displaystyle P_{m,n}^{\gamma}(z, \bar z)$ or by Dunkl \cite{Dunkl83,Dunkl84} and denoted $ \displaystyle R_{m,n}^{(\gamma)}(z)$.
In fact, we have (see Remark \ref{WunscheDunkel}):
$$\mathcal{Z}_{m,n}^\gamma (z,\bar z)  = (\gamma+1)_{m+n} \overline{P_{m,n}^{\gamma}(z, \bar z)}= (\gamma+1)_{m+n} \overline{R_{m,n}^{(\gamma)}(z)} .$$
These polynomials are often referred to as generalized Zernike polynomials as they are related to the real Zernike polynomials  $ \displaystyle R^\nu_k(x)$,
introduced in \cite{ZernikeBrinkman35} and used in the study of diffraction problems, by the relation
$$\mathcal{Z}_{m,n}^0 (z,\bar z) =   (m+n)! e^{i[(n-m)\arg z]}  R^{n-m}_{m+n}(\sqrt{z\bar z}); \quad m\leq n .$$

  An accurate analysis of the basic properties of $ \displaystyle \mathcal{Z}_{m,n}^\gamma (z,\bar z)$, like recurrence relations with respect to the indices $m$ and $n$, differential equations they might obey, some generating functions and so on, have been developed in various papers from different point of views. See \cite{Koornwinder75,Koornwinder78} by Koornwinder for a very nice account on these polynomials and \cite{Wunsch05} for an elegant reintroduction of them. The interested reader can refer also to the recent works \cite{Torre08,Gh08OM,Kanjin2013,Thirulo13,DunklXu14}.
We mention that the disk polynomials appear quite frequently when investigating spectral properties of some special differential operators of Laplacian type acting on $ \displaystyle L^{2,\gamma}(D):=L^2(D;(1-|z|^2)^\gamma d\lambda)$; $d\lambda$ being the Lebesgue measure (see Section 2).
They form a complete orthogonal system (basis) over the Hilbert space $ \displaystyle L^{2,\gamma}(D)$ when $\gamma>-1$.

Even though these polynomials have been known for a long time, little attention has been paid to the operational formulae of Burchnall type in the literature.
Our main objective here is to develop some of these formulae that we will apply in order to obtain some new remarkably interesting identities, including Nielsen identities. We establish also a Runge addition formula involving the disk polynomials. These operational representations are very convenient for obtaining new generating functions which may suggest new applications in physics and combinatorics.

The various results presented in this paper for $ \displaystyle \mathcal{Z}_{m,n}^\gamma (z,\bar z)$ are motivated essentially by those that are obtained in \cite{Gh08,Gh13,Ismail13a}
for the complex Hermite polynomials
\begin{align} \label{CHP}
H_{p,q}(z,\bz )=(-1)^{p+q}e^{z\bz }\dfrac{\partial ^{p+q}}{\partial \bz^{p} \partial z^{q}} \left(e^{-z \bz }\right); \quad z\in\C ,
\end{align}
except that the computations with the disk polynomials  are rather difficult.

The remainder of the paper is organised as follows. In Section 2, we review the factorization method, due to Schr\"odinger,
for the magnetic Laplacian on the unit disk (viewed as a  hyperbolic space) and recall some  properties of the suggested disk polynomials $ \displaystyle \mathcal{Z}_{m,n}^\gamma (z,\bar z)$.
We point out their dependance on the hyperbolic geometry of the disk as well as the physical meaning of the parameter $\gamma$.
In Section 3, we give operational formulae of Burchnall type involving these complex disk polynomials.
Some of their applications are discussed in the next sections.
Indeed, Section 4 deals with recurrence quadratic formulae (called Nielsen identities).
In Section 5, we establish new three-term recurrence formulae with respect to the indices $m,n$ and the continuous parameter $\gamma$.
Section 6 is devoted to the Runge addition formula. In Sections 7 and 8, new generating functions and summation formulae are obtained.

\section{Generation of the complex disk polynomials}
In this section, we review the algebraic method used to generate the disk polynomials, starting from a special magnetic Laplacian $ \displaystyle \mathfrak{L}_\nu$ on the complex unit disk $D\subset \C$
acting on the $ \displaystyle L^2$-Hilbert space $ \displaystyle L^2(D ;(1 - |z|^2)^{-2}d\lambda)$, where $d\lambda(z)= dxdy$ denotes the Lebesgue measure.
It is given by
\begin{align*}
\mathfrak{L}_\nu=     \nabla_{\nu-1}  \nabla_{\nu-1}^{*}  + \nu \quad \mbox{and} \quad
\mathfrak{L}_{\nu -1}= \nabla_{\nu-1}^{*}  \nabla_{\nu-1}  - (\nu-1) ,
 \end{align*}
where the first order differential operator $\nabla_\alpha$ and its formal adjoint
$\nabla_\alpha^{*}$ are given by
\begin{align*}
\nabla_\alpha = -(1 - |z|^2)\frac{\partial}{\partial z} + \alpha \bar z \quad \mbox{and} \quad
\nabla_\alpha^{*}  = (1 - |z|^2)\frac{\partial}{\partial \bar z} + (\alpha +1) z.
\end{align*}
The twisted Laplacian $\mathfrak{L}_\nu$ is an elliptic self-adjoint second order differential operator whose explicit spectral
theory is well-known in the literature \cite{Zhang92,FV}.
In particular, its discrete $L^2$-spectrum is nontrivial if and only if $\nu > 1/2$. In such case, it is known to be given by the eigenvalues $
\lambda_{\nu,m}:=\nu(2m+1) - m(m+1)$ for $m$ being a positive integer such that $0\leq m< \nu -1/2$.
On the other hand, we know since Schr\"odinger that the factorization method  allows one to construct $L^2$-eigenfunctions (see \cite{Sch40/41,FV}).
Indeed, our Laplacian $\mathfrak{L}_\nu$ satisfies the following algebraic relation
  \begin{align*}
\mathfrak{L}_\nu \nabla_{\nu-1} = \left(\nabla_{\nu-1}  \nabla_{\nu-1}^{*} + \nu \right) \nabla_{\nu-1}
                               = \nabla_{\nu-1} \left(  \mathfrak{L}_{\nu -1} +(2\nu-1) \right) .
 \end{align*}
 This implies that the intertwine operator $\nabla_{\nu-1}$ generates eigenfunctions of $\mathfrak{L}_{\nu}$  from those of $\mathfrak{L}_{\nu-1}$.
 Doing so we perform
 \begin{equation}
\nabla^{\nu}_{m} := \nabla_{\nu-1} \circ \nabla_{\nu-2}\circ \cdots \circ \nabla_{\nu - m} = (-1)^m \prod_{j=1}^m \left( (1 - |z|^2)\frac{\partial}{\partial z} -(\nu -j) \bar z \right) \label{mthoder}
\end{equation}
and therefore, if $\varphi_{0}$ is a nonzero $L^2$-eigenfunction of $\mathfrak{L}_{\nu - m}$, then $ \nabla^{\nu}_{m}\varphi_{0}$ is an eigenfunction of $\mathfrak{L}_\nu $ belonging to
$\displaystyle   A^{2,\nu}_m(D):=\{\varphi \in L^2(D ;d\mu);  ~~ \mathfrak{L}_\nu  \varphi =\lambda_{\nu,m} \varphi  \}.$
 Conversely, it is shown in \cite{Gh08OM} that this method describes completely the $L^2$-eigenspaces $A^{2,\nu}_m(D)$.
 More precisely, if for every fixed $\nu>1/2$ and $ 0\leq m < \nu - 1/2$, and varying $n=0,1,2, \cdots $, we consider
  \begin{align}
  \psi_{m,n}^{\nu}(z,\bar z):= \nabla^{\nu}_{m} \left(z^n(1 - |z|^2)^{\nu-m} \right),
\end{align}
 then the functions $ \displaystyle \psi_{m,n}^{\nu}$ are $L^2$-eigenfunctions of $\mathfrak{L}_\nu$. Furthermore, they form
  an orthogonal basis of the $L^2$-eigenspace $ \displaystyle A^{2,\nu}_m(D)$.
 Hence, by rewriting $\nabla_\alpha$ as
   $$   \nabla_{\alpha} f=-(1 - |z|^2)^{1- \alpha}\frac{\partial}{\partial z} [(1 - |z|^2)^{\alpha} f ] ,$$
   we obtain
\begin{align}\label{likeRodrigues}
 \psi_{m,n}^{\nu}(z,\bar z) = (-1)^m (1-|z|^2)^{-\nu} \left[(1 - |z|^2)^{2}
\frac{\partial}{\partial z}\right]^{m}\Big(z^n(1-|z|^2)^{2(\nu-m)}\Big)\,.
\end{align}
This suggests the following class of polynomials in the two variables $z$ and $\bz$:
 \begin{align}
 \Phi_{m,n}^{\gamma}(z, \bar z)  &:= (1 - |z|^2)^{-\nu  +m}\nabla^{\nu}_{m} \left( z^n (1 - |z|^2)^{\nu-m}\right) \label{Nablam}\\
 &= (-1)^m (1-|z|^2)^{-(\gamma+m+1)} \left[(1 - |z|^2)^{2} \frac{\partial}{\partial z}\right]^{m}\Big(z ^{n}(1-|z|^2)^{\gamma+1}\Big) ,\label{KappaPoly0}
  \end{align}
where $\gamma=2(\nu-m)-1$.
More explicitly we have (\cite{Gh08OM}):
\begin{align}
 \Phi_{m,n}^{\gamma}(z, \bar z) = (-1)^{m}  (\minmn )! |z|^{|m-n|}e^{i[(n-m)\arg z]} \mathrm{P}^{(|m-n|,\gamma)}_{\minmn }(1-2 |z|^2), \label{DiscPoly}
 \end{align}
 where $m \wedge n=\min\{m,n\}$ and $ \displaystyle \mathrm{P}^{(\alpha,\beta)}_{k}(x)$ denotes the real Jacobi Polynomials.
In addition, it can be realized in terms of the following differential operator
  \begin{eqnarray} \label{diff-opAm}
  \mathcal{A}_m(f):= (1 -  |z|^2)^{m+1} \frac{\partial^m}{\partial z^m} \left( (1 -  |z|^2)^{m-1} f(z)\right).
\end{eqnarray}
Moreover, a simple induction yields
\begin{proposition} \label{propVeryClosed}
For every sufficiently differentiable function $f$, we have
\begin{eqnarray} \label{Closedh2}
\left[(1 -  |z|^2)^2\frac{\partial}{\partial z}\right]^m (f) =(1 -  |z|^2)^{m+1}
\frac{\partial^m}{\partial z^m} \left( (1 -  |z|^2)^{m-1}f\right) .
\end{eqnarray}
In particular, $ \displaystyle  \mathcal{A}_{m+m'}= \mathcal{A}_{m}\circ \mathcal{A}_{m'}$.
\end{proposition}

According to \eqref{KappaPoly0} and Proposition \ref{propVeryClosed}, we have
  \begin{align} \label{PmnA}
   \Phi_{m,n}^{\gamma}(z,\bar z) =    (-1)^{m}(1 -  |z|^2)^{-\gamma} \frac{\partial^{m}}{\partial z^m} (z^n (1 -  |z|^2)^{\gamma+m})
  \end{align}
which gives rise to the Rodrigues' type formula \cite{Wunsch05,Gh08OM}:
  \begin{align}
  C_{m,n}^\gamma \Phi_{m,n}^{\gamma}(z,\bar z) &=(-1)^{m+n}
  (1 -  |z|^2)^{-\gamma} \frac{\partial^{m+n}}{\partial z^m \partial \bar z ^n}\left((1 -  |z|^2)^{\gamma+m+n}\right) ,  \label{RTF2}
  \end{align}
  where   $  C_{m,n}^\gamma := (\gamma+m+1)_n $  and $(a)_n:=a(a+1) \cdots (a+n-1)$ stands for the Pochhammer symbol.
 From now on, instead of $ \displaystyle \Phi_{m,n}^\gamma (z,\bar z)$, we deal with the polynomials $ \displaystyle \mathcal{Z}_{m,n}^\gamma (z,\bar z) : = C_{m,n}^\gamma \Phi_{m,n}^\gamma (z,\bar z)$,  so that
  \begin{align}
  \mathcal{Z}_{m,n}^\gamma (z,\bar z) &\stackrel{\eqref{Nablam}}{=}
   C_{m,n}^\gamma (1 - |z|^2)^{-\frac{\gamma+1}2}\nabla^{\frac{\gamma+1}2 +m}_{m} \left( z^n (1 - |z|^2)^{\frac{\gamma+1}2}\right)  \\
 & \stackrel{\eqref{PmnA}}{=} (-1)^{m} C_{m,n}^\gamma (1 -  |z|^2)^{-\gamma} \frac{\partial^{m}}{\partial z^m} (z^n (1 -  |z|^2)^{\gamma+m}) \label{2.12}\\
    & \stackrel{\eqref{RTF2}}{=} (-1)^{m+n}  (1 -  |z|^2)^{-\gamma} \frac{\partial^{m+n}}{\partial z^m \partial \bar z^n} \left( (1 -  |z|^2)^{\gamma+m+n} \right).
\end{align}
Note for instance that we have
$$  \overline{\mathcal{Z}_{m,n}^\gamma (z,\bar z)} =  \mathcal{Z}_{m,n}^\gamma (\bar z, z) =  \mathcal{Z}_{n,m}^\gamma (z,\bar z)
\quad \mbox{ and } \quad
\mathcal{Z}_{0,s}^\beta (z,\bar z) = \overline{\mathcal{Z}_{s,0}^\beta (z,\bar z)}=(\beta +1)_s z^s.$$

\section{Operational formulae for the disk polynomials}

In this section we derive new operational formulae for the disk polynomials similar to those obtained in
\cite{Gh13} for the complex Hermite polynomials $ \displaystyle H_{p,q}(z,\bz )$ defined through \eqref{CHP}.
To this end, we need to introduce the following differential operators
  \begin{align}
  \mathcal{A}_{m,n}^\gamma (f)   & :=
   (-1)^m C_{m,n}^\gamma (1 -  |z|^2)^{-\gamma}  \frac{\partial^m}{\partial z^m} \left( z^n (1 -  |z|^2)^{\gamma+m} f \right), \label{diff-opAmnf}\\
   \mathcal{B}_{m,n}^\gamma (f) &:=   (-1)^{m+n}  (1 -  |z|^2)^{-\gamma} \frac{\partial^{m+n}}{\partial z^m \partial \bar z^n} \left( (1 -  |z|^2)^{\gamma+m+n} f \right)  , \label{diff-opZmn2}  \\
    \nabla^{\gamma,\gamma'}_{m,n} (f) &=  \nabla^{\gamma}_{m} \circ  \overline{\nabla}^{\gamma'}_{n} (f), \label{diff-opNablamn}
\end{align}
where
$ \nabla^\alpha_{m} = \nabla_{\alpha-1}\circ\nabla_{\alpha-2}\circ \cdots \circ \nabla_{\alpha-m}$
with
$\nabla_{\alpha}= -(1 -  |z|^2)\frac{\partial}{\partial z} + \alpha \bar z $ and
$\overline{\nabla}^\alpha_{n}  = \overline{\nabla}_{\alpha-1}\circ\overline{\nabla}_{\alpha-2}\circ \cdots \circ \overline{\nabla}_{\alpha-n}$
with
$\overline{\nabla}_{\alpha}=  -(1 -  |z|^2)\frac{\partial}{\partial \bar z} + \alpha z,$
so that
  \begin{align}\label{diff-opZernike2}
  [\mathcal{A}_{m,n}^\gamma(1)] (z)  = [ \mathcal{B}_{m,n}^\gamma (1)] (z)  =   \mathcal{Z}_{m,n}^\gamma (z,\bar z).
\end{align}

The main result of this section is the following.
\begin{theorem} \label{Burch-OpForm}
For given positive integers $m,n$, we have the following operational formulae of Burchnall type involving $ \displaystyle \mathcal{Z}_{m,n}^\gamma (z,\bar z)$:
  \begin{align}
   \mathcal{A}_{m,n}^{\gamma}(f)
 &= m!n! \sum_{j=0}^{m}   \frac{(-1)^{j} (1 -  |z|^2)^{j}}{j!} \frac{ \mathcal{Z}_{m-j,n}^{\gamma+j}(z,\bar{z})}{(m-j)! n!}
  \frac{\partial^{j}}{\partial z^j}(f),
    \label{Burch-OpForm1}\\
   \mathcal{B}_{m,n}^{\gamma}(f)
 &=  m!n! \sum_{j=0}^{m} \sum_{k=0}^{n}
     \frac{(-1)^{j+k} (1 -  |z|^2)^{j+k}}{j!k!}  \frac{ \mathcal{Z}_{m-j,n-k}^{\gamma+j+k}(z,\bar{z})}{(m-j)!(n-k)!}
     \frac{\partial^{j+k}}{\partial z^j \partial \bar z^k} (f),
    \label{Burch-OpForm2}\\
   \nabla^{\gamma,\gamma'}_{m,n} (f)
  &=   m!n! \sum_{j=0}^{m}\sum_{k=0}^{n}   \frac{(\gamma'+ k - n)_{n-k}}{(\gamma+k)_{n-k}} \label{Burch-OpFormNabla} \\
  &  \qquad \qquad  \frac{(-1)^{j+k} (1 -  |z|^2)^{j+k}}{j!k!} \frac{\mathcal{Z}_{m-j,n-k}^{(\gamma-1-m)+k+j}(z,\bar{z})}{(m-j)!(n-k)!} \frac{\partial^{k+j}}{\partial z^{j}\partial \bar{z}^{k}}(f),      \nonumber
\end{align}
for every sufficiently differentiable function $f$.
\end{theorem}

\begin{remark}
Similar formulae can be obtained using
$\sum\limits_{k=0}^m A_k = \sum\limits_{k=0}^m A_{m-k}.$
\end{remark}

\begin{proof}
We start from \eqref{diff-opAmnf} and we make use of the Leibnitz formula to get
  \begin{align*}
  \mathcal{A}_{m,n}^\gamma(f)
   = (-1)^m C_{m,n}^\gamma (1 -  |z|^2)^{-\gamma} \sum_{j=0}^{m} \binom{m}{j}  \left( \frac{\partial^{m-j}}{\partial z^{m-j}} ( z^n (1 -  |z|^2)^{\gamma+m})\right)
  \left( \frac{\partial^j}{\partial z^j} f \right).
 \end{align*}
The statement follows from  \eqref{diff-opZernike2}  since $C_{m,n}^\gamma=C_{m-j,n}^{\gamma+j}$.

The proof of \eqref{Burch-OpForm2} is quite similar by repeated application of the Leibnitz formula to \eqref{diff-opZmn2}
combined with the fact that
\begin{align*}
\frac{\partial^{(m-j)+(n-k)}}{\partial z^{m-j} \partial \bar z^{n-k}}( (1 -  |z|^2)^{\gamma+m+n})
 =  (-1)^{(m-j)+(n-k)} (1 -  |z|^2)^{\gamma+j+k}   \mathcal{Z}_{m-j,n-k}^{\gamma+j+k}(z,\bar{z}) .
 \end{align*}

To prove \eqref{Burch-OpFormNabla}, notice first that we have
\begin{align*}
&\nabla^\gamma_{m} f = (-1)^m (1 -  |z|^2)^{-\gamma} \left[(1 -  |z|^2)^2\frac{\partial}{\partial z}\right]^m ((1 -  |z|^2)^{\gamma-m} f)
\\
&\overline{\nabla}^{\gamma'}_{n} f = (-1)^n (1 -  |z|^2)^{-\gamma'} \left[(1 -  |z|^2)^2\frac{\partial}{\partial \bz}\right]^{n}((1 -  |z|^2)^{\gamma' -n} f) .
\end{align*}
By applying twice Proposition \ref{propVeryClosed}, we get
\begin{align*} \nabla^{\gamma,\gamma'}_{m,n} f
=(-1)^{m+n}(1 -  |z|^2)^{m+1-\gamma}\frac{\partial^{m}}{\partial z^{m}}\left((1 -  |z|^2)^{\gamma-\gamma'+n}\frac{\partial^{n}}{\partial\bar{z}^{n}}((1 -  |z|^2)^{\gamma'-1}f)\right).
\end{align*}
Finally, \eqref{Burch-OpFormNabla} follows by means of Leibnitz formula and the facts that $ \displaystyle (-a)_k=(-1)^k(a-k+1)_k$ and
$ \frac{\partial^{k}}{\partial z^k}((1 -  |z|^2)^{\beta}) = (-\beta)_k {\bar z}^k (1 -  |z|^2)^{\beta-k}$, keeping in mind the expression of $\mathcal{Z}_{m,n}^\gamma (z,\bar z)$ given through \eqref{2.12}.
\end{proof}

An immediate consequence of Theorem \ref{Burch-OpForm}, is the following

\begin{corollary} We have the following identities
 \begin{align}
 & \sum_{j=0}^{m} \frac{(-m)_{j} (\gamma+m)_j }{j!}   {\bar z}^{j} \mathcal{Z}_{m-j,n}^{\gamma+j}(z,\bar{z})  =0 ; \quad m>n, \label{BurchnallConsq1}\\
 & \sum_{j=0}^{m}    \frac{(-m)_{j} (\gamma+m)_j}{j!}  {\bar z}^{j}  \mathcal{Z}_{m-j,n}^{\gamma+j}(z,\bar{z})
   =    (-n)_{m} (\gamma+1+m)_n    z^{n-m} (1-|z|^2)^m  ; \quad m\leq n . \nonumber
\end{align}
\end{corollary}

\begin{proof}
The identities follow easily
from  \eqref{Burch-OpForm1} and \eqref{diff-opAmnf} with $ \displaystyle f(z) = (1 -  |z|^2)^{-\gamma-m}$ and making use of $ \displaystyle \frac{ (-1)^{j}}{(m-j)! } = \frac{(-m)_j}{m!}$. In fact, we have
$$\mathcal{A}_{m,n}^{\gamma}((1 -  |z|^2)^{-\gamma-m})=  \dfrac{(-1)^m n! }{(n-m)!} (\gamma+m+1)_n z^{n-m} (1 -  |z|^2)^{-\gamma} $$
when $ n\geq m$ and $ \displaystyle \mathcal{A}_{m,n}^{\gamma}((1 -  |z|^2)^{-\gamma-m})=0$ otherwise.
\end{proof}

\begin{remark}
    The identity \eqref{BurchnallConsq1} for $n=0$ is a special case of the Chu-Vandermonde identity
    $${_2F_1}\left( \begin{array}{c} -k , b \\ c \end{array}\bigg | 1 \right) =\frac{(c-b)_k}{(c)_k}; \quad k \in \Z^+\,.$$
 Indeed, since $ (\gamma+j+1)_{m-j} ={(\gamma+1)_{m}}/{(\gamma+1)_{j}} $,  the left hand side of \eqref{BurchnallConsq1} reduces to
      \begin{align*}
  (\gamma+1)_m {\bar z}^{m} \sum_{j=0}^{m} \frac{(-m)_{j}  (\gamma+m)_j}{ j!(\gamma+1)_j } &=
   (\gamma+1)_m {\bar z}^{m} {_2F_1}\left( \begin{array}{c} -m , \gamma+m \\ \gamma+1 \end{array}\bigg | 1 \right)
   \\& = {\bar z}^{m} (1-m)_m
   =0,
\end{align*}
for every fixed integer $m>0=n$.
\end{remark}

The complex Hermite polynomials can be  retrieved from disk polynomials $ \displaystyle \mathcal{Z}_{m,n}^{\gamma}(z,\bar{z})$ by an appropriate limiting process $\gamma \to +\infty$ (see \cite{Gh08OM}). The following result  gives a direct relationship between $ \displaystyle  H_{m,n}(z, \bar z)$ and $ \displaystyle \mathcal{Z}_{m,n}^{\gamma}(z,\bar{z})$.

\begin{corollary} The complex Hermite polynomials $ \displaystyle H_{m,n}(z, \bar z)$  restricted to the unit disk
are given in terms of the disk polynomials by
 \begin{align*}
  H_{m,n}(z, \bar z) = \frac{m!(1 -  |z|^2)^{-m}}{(\gamma+m+1)_n}  \sum_{j=0}^{m} \sum_{k=0}^{j}
   \frac{(-1)^k(\gamma+m)_k}{k!}  \frac{{\bar z}^{j} (1 -  |z|^2)^{j-k}}{(j-k)!} \frac{ \mathcal{Z}_{m-j,n}^{\gamma+j}(z,\bar{z})}{(m-j)! } .
\end{align*}
\end{corollary}

\begin{proof}
Notice first that $$\mathcal{A}_{m,n}^{\gamma}((1 -  |z|^2)^{-\gamma-m}e^{-|z|^2})=
(\gamma+m+1)_n (1 -  |z|^2)^{-\gamma}e^{-|z|^2}  H_{m,n}(z, \bar z).$$
The right hand side of \eqref{Burch-OpForm1}, with $ \displaystyle f(z) = (1 -  |z|^2)^{-\gamma-m}e^{-|z|^2}$, reads
\begin{align*}
  m! (1 -  |z|^2)^{-\gamma-m}  \sum_{j=0}^{m} \sum_{k=0}^{j}
   \frac{(-1)^k(\gamma+m)_k}{k!}  \frac{{\bar z}^{j} (1 -  |z|^2)^{j-k}}{(j-k)!} \frac{ \mathcal{Z}_{m-j,n}^{\gamma+j}(z,\bar{z})}{(m-j)! } .
\end{align*}
Thus, the result follows from \eqref{Burch-OpForm1} by straightforward computation making use of the Leibnitz rule.
 \end{proof}

Some applications of the  operational formulae obtained above will be discussed in the following sections.

\section{Quadratic recurrence formula}
We consider various expressions for the function $f$ that leads to some new and interesting identities.
We start with $f=z^s$, then by \eqref{Burch-OpForm1}, we get
\begin{align}
   \mathcal{Z}_{m,n+s}^{\gamma}(z,\bar z)
   = m!n!s! C_{m+n,s}^{\gamma} \sum_{j=0}^{\minms }
    \frac{(-1)^{j} z^{s-j} (1 -  |z|^2)^{j}}{(s-j)! j!} \frac{ \mathcal{Z}_{m-j,n}^{\gamma+j}(z,\bar{z})}{(m-j)! n!},
    \label{Burch-OpForm1s}
\end{align}
where $\minms =\min\{m,s\}$. Indeed, this follows from taking into account the fact that
  \begin{align}
    \mathcal{Z}_{m,n+s}^{\gamma+\beta} (z,\bar z) = \dfrac{C_{m,n+s}^{\gamma+\beta}}{C_{m,n}^\gamma} \mathcal{A}_{m,n}^{\gamma}(z^s (1 -  |z|^2)^{\beta}) .
\end{align}
Moreover, for $ \displaystyle f= (1 -  |z|^2)^{s}$, we obtain
\begin{align}
   \mathcal{Z}_{m,n}^{\gamma+s}(z,\bar z)= m!n!s! \frac{C_{m,n}^{\gamma+s}}{C_{m,n}^{\gamma}} \sum_{j=0}^{\minms }
   \frac{{\bar z}^{j}}{(s-j)!j!}  \frac{ \mathcal{Z}_{m-j,n}^{\gamma+j}(z,\bar{z})}{(m-j)! n!}.
    \label{Burch-OpForm1beta}
\end{align}

We now  derive further identities of Nielsen type for the polynomials $ \displaystyle  \mathcal{Z}_{m,n}^{\gamma}(z,\bar z)$. Namely, we state the following

\begin{theorem}
For every fixed positive integers $ m,n, r$ and $s$, we have the following quadratic recurrence formulae
\begin{align}
  \frac{\mathcal{Z}_{m+r,n+s}^{\gamma}(z,\bar{z})}{m!n!r!s!} &=  \sum_{j=0}^{\minms} \sum_{k=0}^{\minnr}  (\gamma+m+n+r+1)_j(\gamma+m+n+s+1)_k \label{Nielsen1} \\
    & \times \frac{(-1)^{j+k} (1 -  |z|^2)^{j+k}}{j!k!}  \frac{ \mathcal{Z}_{m-j,n-k}^{\gamma+j+k}(z,\bar{z})}{(m-j)!(n-k)!}
     \frac{\mathcal{Z}_{r-k,s-j}^{\gamma+m+n+j+k}(z,\bar{z})}{(s-j)!(r-k)!} \nonumber
\end{align}
and
\begin{align}\label{Nielsen2}
   \frac{\mathcal{Z}_{m+r,n}^{\gamma}(z,\bar{z})}{m!n!}
     =    \sum_{j=0}^{\minmn} (\gamma+m+r+1)_j  (\gamma+j+1)_{m-j}\frac{(-1)^{j} {\bar z}^{m-j} (1 -  |z|^2)^{j}}{(m-j)! j!}
       \frac{\mathcal{Z}_{r,n-j}^{\gamma+m+j}(z,\bar{z})}{(n-j)!}.
\end{align}
\end{theorem}

 \begin{proof} From the fact
 $ \displaystyle \mathcal{Z}_{r,s}^{\beta}(z,\bar{z}) = \mathcal{B}_{r,s}^{\beta}(1)$ and Definition \ref{diff-opZmn2} of $ \displaystyle \mathcal{B}_{m,n}^{\gamma}(f)$, we easily verify that $ \displaystyle  \mathcal{Z}_{m+r,n+s}^{\gamma}(z,\bar{z}):= \mathcal{B}_{m,n}^{\gamma}(\mathcal{Z}_{r,s}^{\gamma+m+n}(z,\bar{z})).$
 Thus, from the operational representation \eqref{Burch-OpForm2} combined with the fact
 \begin{equation} \label{derivee}
 \frac{\partial^{j+k}}{\partial z^j \partial \bar z^k} \left(\mathcal{Z}_{r,s}^{\alpha}(z,\bar z)\right)
 = \frac{r!s!(\alpha+r+1)_j(\alpha+s+1)_k}{(r-k)!(s-j)!}\mathcal{Z}_{r-k,s-j}^{\alpha+j+k}(z,\bar{z})
\end{equation}
for $j\leq s$ and $k\leq r$, we get
\begin{align*}
  \mathcal{Z}_{m+r,n+s}^{\gamma}(z,\bar{z})
    =
m!n!r!s! \sum_{j=0}^{\minms } \sum_{k=0}^{\minnr } & (\gamma+m+n+r+1)_j(\gamma+m+n+s+1)_k\frac{(-1)^{j+k} (1 -  |z|^2)^{j+k}}{j!k!}\\
& \times  \frac{ \mathcal{Z}_{m-j,n-k}^{\gamma+j+k}(z,\bar{z})}{(m-j)!(n-k)!}\frac{\mathcal{Z}_{r-k,s-j}^{\gamma+m+n+j+k}(z,\bar{z})}{(s-j)!(r-k)!}\,.
\end{align*}
The proof of \eqref{Nielsen2} lies essentially on the observation that
 $ \displaystyle \mathcal{Z}_{m+r,n}^{\gamma}(z,\bar{z}):= \mathcal{A}_{m,0}^{\gamma}(\mathcal{Z}_{r,n}^{\gamma+m}(z,\bar{z})),$
which can be handled from the representation \eqref{diff-opAmnf} of the operator $ \displaystyle \mathcal{A}_{m,0}^{\gamma}(f)$ together with the fact that
 $ \displaystyle \mathcal{Z}_{r,n}^{\beta}(z,\bar{z}) = \mathcal{A}_{r,n}^{\beta}(1)$.
Next, the operational formula \eqref{Burch-OpForm1} infers
  \begin{align*}
   \mathcal{Z}_{m+r,n}^{\gamma}(z,\bar{z})
  &= m!n!  \sum_{j=0}^{\minmn } (\gamma+m+r+1)_j \frac{(-1)^{j} (1 -  |z|^2)^{j}}{j!} \frac{ \mathcal{Z}_{m-j,0}^{\gamma+j}(z,\bar{z})}{(m-j)! }
       \frac{\mathcal{Z}_{r,n-j}^{\gamma+m+j}(z,\bar{z})}{(n-j)!}     .
  \end{align*}
  This leads to \eqref{Nielsen2} as $ \displaystyle \mathcal{Z}_{0,s}^{\alpha}(z,\bar{z})= (\alpha+1)_sz^s$.
\end{proof}

\begin{remark}
We recover  \eqref{Burch-OpForm1s} from \eqref{Nielsen1} by taking $s=0$.
For the specific values  $r=0$ in \eqref{Nielsen2} or $n=0$ in \eqref{Burch-OpForm1s}, we get the explicit expression of
the polynomials $ \displaystyle \mathcal{Z}_{m,n}^{\gamma}(z,\bar{z})$:
   \begin{align*}
   \mathcal{Z}_{m,n}^{\gamma}(z,\bar{z})
 =  m!n!  \sum_{j=0}^{\minmn } \frac{\Gamma(\gamma+m+n+1)}{\Gamma(\gamma+j+1)}
    \frac{(-1)^{j} (1 -  |z|^2)^{j}}{j!} \frac{\bar z^{m-j}}{(m-j)! } \frac{z^{n-j} }{(n-j)!}\,.
  \end{align*}
\end{remark}

\begin{remark}\label{WunscheDunkel}
The explicit expression above of $ \displaystyle \mathcal{Z}_{m,n}^{\gamma}(z,\bar{z})$ can be rewritten easily
in terms of the Gauss hypergeometric function ${_2F_1}$ as (\cite[p. 137]{Wunsch05})
     \begin{align*}
   \mathcal{Z}_{m,n}^{\gamma}(z,\bar{z})  =  (\gamma+1)_{m+n} \overline{z}^m z^n
   {_2F_1}\left( \begin{array}{c} -m , -n \\ \gamma+1 \end{array}\bigg | 1-\frac 1{|z|^2} \right)=(\gamma+1)_{m+n} \overline{P_{m,n}^{\gamma}(z, \bar z)}
  \end{align*}
 from which we can recover the well-known ${_2F_1}$ formula for the disk polynomials (see \cite[p. 692]{Dunkl83} or \cite[p. 535]{Dunkl84}),
     \begin{align*}
   \mathcal{Z}_{m,n}^{\gamma}(z,\bar{z})  =  \frac{\left((\gamma+1)_{m+n}\right)^2} {(\gamma+1)_{m}(\gamma+1)_{n}}\overline{z}^m z^n
   {_2F_1}\left( \begin{array}{c} -m , -n \\ -\gamma-m-n \end{array}\bigg | \frac 1{|z|^2} \right) = (\gamma+1)_{m+n}  \overline{R_{m,n}^{(\gamma)}(z)}.
  \end{align*}
\end{remark}

\section{Three-term recurrence relations}
The well known three-term recurrence relations for the disk polynomials can be obtained starting from their explicit expression in $z$ and $\bar z$ (see for example \cite{Koornwinder75,Koornwinder78,Dunkl83,Wunsch05,DunklXu14}). The basic one reads in terms of $ \displaystyle \mathcal{Z}_{m,n}^{\gamma}(z,\bar z)$ as
\begin{align*}
\left(  \frac{\gamma+m+1}{\gamma+m+n+1}\right) &\mathcal{Z}_{m+1,n}^{\gamma}(z,\bar{z})
\\&= \left( {\gamma+m+n+1}\right) \bar z \mathcal{Z}_{m,n}^{\gamma}(z,\bar{z})  - n   \left({\gamma+m+n}\right)\mathcal{Z}_{m,n-1}^{\gamma}(z,\bar{z})
   \end{align*}
involving the same argument $\gamma$. Below, we establish new recurrence relations for $ \displaystyle \mathcal{Z}_{m,n}^{\gamma}$ with respect to the indices $m,n$ and the
argument $\gamma$.

\begin{theorem}  We have the following  recursive relations with different arguments
\begin{align}
 & \mathcal{Z}_{m,n+1}^{\gamma}(z,\bar{z}) = (\gamma+m+n+1) \left\{ z \mathcal{Z}_{m,n}^{\gamma}(z,\bar{z}) - m (1 -  |z|^2) \mathcal{Z}_{m-1,n}^{\gamma+1}(z,\bar{z})\right\}, \label{TTR11} \\
& \mathcal{Z}_{m,n}^{\gamma}(z,\bar{z}) = \left( \frac{\gamma+m+1}{\gamma+m+n+1}\right) \mathcal{Z}_{m,n}^{\gamma+1}(z,\bar{z}) -  m \bar z \mathcal{Z}_{m-1,n}^{\gamma+1}(z,\bar{z}) \label{TTR21},\\
 &  \mathcal{Z}_{m,n}^{\gamma}(z,\bar{z}) = \frac{1}{\gamma+m+n+1}
 \left\{\bar z \mathcal{Z}_{m,n+1}^{\gamma}(z,\bar{z}) + (\gamma+m+1) (1 -  |z|^2) \mathcal{Z}_{m,n}^{\gamma+1}(z,\bar{z}) \right\} ,\label{TTR41} \\
 & \mathcal{Z}_{m+1,n}^{\gamma}(z,\bar z) = (\gamma+1) \bar z \mathcal{Z}_{m,n}^{\gamma+1}(z,\bar z) -
  n (\gamma+m+2)  (1 -  |z|^2) \mathcal{Z}_{m,n-1}^{\gamma+2}(z,\bar z)). \label{TTR51}
\end{align}
\end{theorem}

\begin{proof}
  \eqref{TTR11} is a special case of \eqref{Burch-OpForm1s} by taking $s=1$. It can also be derived directly from  \eqref{diff-opZernike2} by applying again   the derivative rule to
  $$\mathcal{Z}_{m+1,n}^{\gamma}(z,\bar{z}) =
  (-1)^{m+n} C^\gamma_{m,n} (1 -  |z|^2)^{-\gamma}\dfrac{\partial^m}{\partial z^m}\left(\dfrac{\partial}{\partial z}\left( z^n (1 -  |z|^2)^{\gamma+m+1}\right)\right)
  .$$
  The recurrence relation \eqref{TTR21} is a special case of \eqref{Burch-OpForm1beta} with $s=1$.
 \eqref{TTR41} follows by linear combinations of \eqref{TTR11} and \eqref{TTR21}.
 While \eqref{TTR51} is an immediate consequence of the well established fact
\begin{align}
  (1 -  |z|^2) \frac{\partial}{\partial z} (\mathcal{Z}_{m,n}^{\gamma+1}(z,\bar z)) = (\gamma+1) \bar z \mathcal{Z}_{m,n}^{\gamma+1}(z,\bar z) -  \mathcal{Z}_{m+1,n}^{\gamma}(z,\bar z) \label{TTR31}
\end{align}
combined with \eqref{derivee} for $j=1$ and $k=0$. \eqref{TTR31} is checked by writing $ \displaystyle \mathcal{Z}_{m+1,n}^{\gamma}(z,\bar{z})=\mathcal{A}_{m+1,n}^{\gamma}(1)$  as
  $$
  \mathcal{Z}_{m+1,n}^{\gamma}(z,\bar{z})= -(1 -  |z|^2)^{-\gamma}\dfrac{\partial}{\partial z}\left((1 -  |z|^2)^{\gamma+1}\mathcal{Z}_{m,n}^{\gamma+1}(z,\bar{z})\right)
  $$
   and next applying the derivative rule to the involved product.
\end{proof}

\begin{remark}
The basic three term recurrence relation mentioned in the beginning of this section appears as an immediate consequence of combining \eqref{TTR11} and \eqref{TTR41}.
Moreover, linear combinations of \eqref{TTR21} and \eqref{TTR31} give rise to the following
\begin{align}
\left\{(\gamma+m+1) \bar z - (1 -  |z|^2) \right\} \mathcal{Z}_{m,n}^{\gamma}(z,\bar{z}) =  \left(\frac{\gamma+m+1}{\gamma+m+n+1} \right) \mathcal{Z}_{m+1,n}^{\gamma}(z,\bar{z}) \label{TTR61}
\end{align}
which is equivalent to Eq. (6.2) in \cite{Wunsch05}.
\end{remark}

\begin{remark}
 The conjugate counterparts of the recurrence formulae \eqref{TTR11}, \eqref{TTR21}, \eqref{TTR41} and \eqref{TTR51} follow upon using
  $ \displaystyle \overline{\mathcal{Z}_{m,n}^\gamma (z,\bar z)} =  \mathcal{Z}_{n,m}^\gamma (z,\bar z)$.
\end{remark}

\section{A Runge's addition formula for disk polynomials}
In this section, we prove a Runge's addition formula for the disk polynomials. This formula, when used in its extended form, will help develop
a test for the assumption of bivariate normality (see \cite{Phipps71}).

\begin{theorem}
We have the following addition formula
\begin{align}
&\left(1 - \left|\frac{z+w}{\sqrt{2}}\right|^2\right)^{\gamma} \mathcal{Z}_{m,n}^{\gamma}\left(\frac{z+w}{\sqrt{2}},\frac{\overline{z+w}}{\sqrt{2}}\right)
=   \left(\frac{1}{2}\right)^{\gamma+m+\frac {m+n}2}  m!n!(\gamma+m+n)!   \label{RungeZ}\\
&\times \sum_{j=0}^{m} \sum_{k=0}^{n} \sum_{|\mathbf{s}|= \gamma+m} \frac{(-1)^{s_{3}+s_{4}}
\bar{z}^{s_{4}}\bar{w}^{s_{3}} h^{s_{1}-j}(z) h^{s_{2}-m+j}(w)}{\mathbf{s}! j! k!(m-j)!(n-k)!}  \frac{ \mathcal{Z}_{j,s_{3}+k}^{s_{1}-j}(z,\bar{z}) \mathcal{Z}_{m-j,s_{4}+n-k}^{s_{2}-m+j}(w,\bar{w})}{   (s_1+1)_{s_3+k}(s_2+1)_{s_4+n-k}   } \nonumber
\end{align}
for every $|z|<1$ and $|w|<1$ with $|z+w|<\sqrt{2}$,  where $\gamma +m$ is assumed to be a positive integer.
Here $ \displaystyle \binom{m+\gamma}{\mathbf{s}} :={(\gamma+m)! }/{\mathbf{s}!}$, $\mathbf{s}!=s_1!s_2!s_3!s_4!$ and $|\mathbf{s}|=s_{1}+s_{2}+s_{3}+s_{4}$ for $s=(s_1,s_2,s_3,s_4)$; $s_j \in \Z^{+}$.
\end{theorem}

\begin{proof} Since $$  \mathcal{Z}_{m,n}^{\gamma}(z,\bar{z}) = (-1)^m  (\gamma+m+1)_n (1 -  |z|^2)^{-\gamma}\frac{\partial^{m}}{\partial z^m}(z^n (1 -  |z|^2)^{\gamma+m}),$$
we can write the left hand side of \eqref{RungeZ} as
\begin{align*}
 (-1)^m  (\gamma+m+1)_n \frac{\partial^m}{\partial \left(\frac{z+w}{\sqrt{2}}\right)^m}
\left(\left(\frac{z+w}{\sqrt{2}}\right)^n \left(1 - \left|\frac{z+w}{\sqrt{2}}\right|^2\right)^{\gamma + m}\right).
\end{align*}
Making use of the facts that
$$\frac{\partial}{\partial \left(\frac{z+w}{\sqrt{2}}\right)}= \frac 1{\sqrt{2}}  \left(\frac{\partial}{\partial z} +
\frac{\partial}{\partial w}\right) \quad \mbox { and } 
 (1 - |(z+w)/\sqrt{2}|^2) =  \frac 12 \left( (1 -  |z|^2) + (1 -  |w|^2) - z\bar w - \bar z w\right)$$
as well as the binomial formulae, including $ \displaystyle (X_{1}+X_{2}+X_{3}+X_{4})^{m}=\sum\limits_{|\mathbf{s}|=m}
\binom{m}{\mathbf{s}} X_{1}^{s_{1}}X_{2}^{s_{2}}X_{3}^{s_{3}}X_{4}^{s_{4}}$, the above expression reduces further to
\begin{align*}
(-1)^m \left(\frac{1}{2}\right)^{\gamma+2m+\frac {m+n}2}&
(\gamma+m+1)_n \sum_{j=0}^{m} \sum_{k=0}^{n} \sum_{|\mathbf{s}|= \gamma+m} (-1)^{s_{3}+s_{4}}
\binom{m}{j} \binom{n}{k} \binom{m+\gamma}{\mathbf{s}} \bar{z}^{s_{4}}\bar{w}^{s_{3}}
\\&
\times \frac{\partial^j}{\partial z^j} \left( z^{s_{3}+k} (1 -  |z|^2)^{s_{1}} \right)  \frac{\partial^{m-j}}{\partial w^{m-j}} \left( w^{s_{4}+n-k} (1 -  |w|^2)^{s_{2}}\right)
\end{align*}
that we can rewrite in terms of $ \displaystyle  \mathcal{Z}_{m,n}^{\gamma}(z,\bar{z})$  to obtain
\eqref{RungeZ}.
\end{proof}

\begin{remark}
Under the assumption $\gamma + m$ is a positive integer and by writing down \eqref{RungeZ} for the particular
 case of $z+w=0$, $m=n=0$, we get the identity
\begin{align*}
  \sum_{s_{1}+s_{2}+s_3+s_4= \gamma} \frac{ |z|^{2(s_3+s_4)} (1-|z|^2)^{s_{1}+s_{2}}}{s_{1}!s_{2}!s_3!s_4!}= \frac{2^{\gamma}}{\gamma!}.
\end{align*}
\end{remark}

\section{Generating functions}

In this section, we establish new generating functions involving disk polynomials as well as some of their direct consequences.

\begin{theorem} We have the following generating functions
 \begin{align}\label{GenFctZernike1}
  \sum_{n=0}^\infty \frac{v^n}{n!} \mathcal{Z}_{m,n}^\gamma(z,\bar z)
  =   m! \frac{{\bar z}^{m}}{\left(1-v z\right)^{\gamma+m+1}}  P^{(\gamma, 0)}_m\left( 1 - 2 \frac{v(1- z \bar z)}{\bar z(1-v z)} \right)
\end{align}
for every fixed $v\in\C$ such that $|v|<1$,  and
  \begin{align*}
  \sum_{m=0}^\infty \sum_{n=0}^\infty \frac{u^m}{m!}\frac{v^n}{n!} \mathcal{Z}_{m,n}^\gamma(z,\bar z)
  =    \left(\frac 1{1-v z-u\bar z}\right)^{\gamma+1}
 {_2F_1}\left( \begin{array}{c} \frac{\gamma+1}2 , \frac{\gamma+2}2 \\ \gamma+1 \end{array}
 \bigg | -  \frac{4u v(1-z\bar z)}{(1-v z-u\bar z)^2}\right)
 \end{align*}
 for every fixed $u,v\in\C$ such that $|u|<1$, $|v|<1$ and $|u|+|v|<1$.
 The involved series in $z$ converge absolutely and uniformly on compact sets of the unit disc.
\end{theorem}

\begin{proof} Starting from \eqref{diff-opZernike2}
and using the fact that the power series $ \displaystyle \sum\limits_{n\geq 0}^\infty (\gamma+m+1)_n\frac{(vz)^n}{n!}$ converges
absolutely and uniformly on compact sets of the unit diskfor every fixed $|v|<1$, as well as
the fact that $ \displaystyle \sum\limits_{n=0}^{+\infty} \dfrac{(a)_n}{n!} x^n= (1-x)^{-a}$; $|x|<1$, we get
 \begin{align*}
  \sum_{n=0}^\infty \frac{v^n}{n!} \mathcal{Z}_{m,n}^\gamma(z,\bar z)
 & =  (-1)^m (1 -  |z|^2)^{-\gamma} \frac{\partial^m}{\partial z^m}\left( \sum_{n\geq 0}^\infty (\gamma+m+1)_n \frac{(vz)^n}{n!} (1 -  |z|^2)^{\gamma+m} \right) \\
 &  = (-1)^m (1 -  |z|^2)^{-\gamma}\frac{\partial^m}{\partial z^m}\left( (1-vz)^{-(\gamma+1+m)} (1-z\bar z)^{\gamma+m} \right).
\end{align*}
Now, from the well known fact that
\begin{align*}
\frac{\partial^m}{\partial t^m} \left( (1-xt)^{\alpha} (1-yt)^{\beta} \right)
&= (-1)^m m! y^m (1-xt)^{\alpha} (1-yt)^{\beta-m}\\
&\qquad \times
P^{(\beta-m, -\alpha - \beta -1)}_m\left( 1 - 2 \frac{x(1-yt)}{y(1-xt)} \right),
\end{align*}
with $\alpha= -(\gamma+1+m)$, $\beta = \gamma+m$, $x=v$, $y=\bar z$ and $t=z$, we have
\begin{align*}
  \sum_{n=0}^\infty \frac{v^n}{n!} \mathcal{Z}_{m,n}^\gamma(z,\bar z)
=   m! \left(\frac{1}{1-v z}\right)^{\gamma+1} \left(\frac{\bar z}{1-v z}\right)^{m}
 P^{(\gamma, 0)}_m\left( 1 - 2 \frac{v(1- z \bar z)}{\bar z(1-v z)} \right) .
\end{align*}
 Therefore, it follows that
 \begin{align*}
  \sum_{m=0}^\infty \sum_{n=0}^\infty \frac{u^m}{m!}\frac{v^n}{n!} \mathcal{Z}_{m,n}^\gamma(z,\bar z)
 &=    \left(\frac{1}{1-v z}\right)^{\gamma+1} \sum_{m=0}^\infty \left(\frac{u\bar z}{1-v z}\right)^{m}
 P^{(\gamma, 0)}_m\left( 1 - 2 \frac{v(1- z \bar z)}{\bar z(1-v z)} \right) .
\end{align*}
The series in $z$ converges absolutely and uniformly on compact sets of the unit disc, for every fixed $|u|,|v|<1$ with $|u|+|v|<1$, since in this case
$ \left|\frac{u\bar z}{1-v z}\right| < 1$
and
$$ \left| P^{(\gamma, 0)}_m\left( 1 - 2 \frac{v(1- z \bar z)}{\bar z(1-v z)} \right) \right|
 \leq \left\{ \begin{array}{ll}  1 & \quad -1 <\gamma \leq 0 \\ \frac{(\gamma+1)_m}{m!}  & \quad \gamma \geq 0 \end{array} \right. .$$
Moreover, by means of the generating function \cite[p.256]{Rainville71}, namely
$$ \sum_{m=0}^\infty \frac{(\alpha+\beta+1)_m}{(\alpha+1)_m}t^m P^{(\alpha,\beta)}_m(x)=
 (1-t)^{-(\alpha+\beta+1)} {_2F_1}\left( \begin{array}{c} \frac{\alpha+\beta+1}2 , \frac{\alpha+\beta+2}2 \\ \alpha+1 \end{array}
 \bigg | 2\frac{t(x-1)}{(1-t)^2} \right),
 $$
with
$t=\frac{u\bar z}{1-v z}$, $\alpha=\gamma , \quad \beta= 0$ and $x=1 - 2 \frac{v(1- z \bar z)}{\bar z(1-v z)},$
 it follows that
  \begin{align*}
  \sum_{m=0}^\infty \sum_{n=0}^\infty \frac{u^m}{m!}\frac{v^n}{n!} \mathcal{Z}_{m,n}^\gamma(z,\bar z)
 =    \left(\frac 1{1-v z-u\bar z}\right)^{\gamma+1}
 {_2F_1}\left( \begin{array}{c} \frac{\gamma+1}2 , \frac{\gamma+2}2 \\ \gamma+1 \end{array}
 \bigg | - \frac{4u v(1-z\bar z)}{(1-v z-u\bar z)^2}\right).
 \end{align*}
\end{proof}

Consequently, we have

\begin{corollary} \label{cor7.2} The following summation formulae hold
  \begin{align}
& \sum_{n=0}^{+\infty} \frac{{\bar z}^{n}}{n!} \mathcal{Z}_{m,n}^{\gamma}(z,\bar{z})= (-1)^m m! {\bar z}^m \left(1-|z|^2\right)^{-(\gamma+m+1)}
\label{GenZernike1}\\
& \sum_{m,n=0}^{+\infty} \frac{z^m{\bar z}^{n}}{m!n!}\mathcal{Z}_{m,n}^{\gamma}(z,\bar{z}) = \left(1-|z|^2\right)^{-\gamma}
  \label{GenZernike2}\\
& \sum_{m,n=0}^{+\infty} \frac{(1 -  |z|^2)^{m}{\bar z}^{n}}{m!n!}\mathcal{Z}_{m,n}^{\gamma}(z,\bar{z})= \frac{ \left(1-|z|^2\right)^{-(\gamma+1)}}{1+\bar z}\,.
\label{GenZernike4}
\end{align}
\end{corollary}

\begin{proof}  It is worth observing that \eqref{GenZernike1} follows easily from \eqref{GenFctZernike1} by setting  $v=\bar z$ and $u =z$ therein and using the well known fact that $ \displaystyle P_m^{(\alpha,\beta)}(-1) = \frac{(-1)^m(1+\beta)_m}{n!}$.
For the same specification, the left hand side of \eqref{GenZernike2} reads as
  \begin{align*}
\sum_{m,n=0}^{+\infty} \frac{z^m{\bar z}^{n}}{m!n!}\mathcal{Z}_{m,n}^{\gamma}(z,\bar{z}) =  \left(1-2|z|^2\right)^{-(\gamma+1)}
  {_2F_1}\left( \begin{array}{c} \frac{\gamma+1}2 , \frac{\gamma+2}2 \\ \gamma+1 \end{array}\bigg | -\frac{4|z|^2(1-|z|^2)}{(1-2|z|^2)^2}\right)\,.
  \end{align*}
   Whence, \eqref{GenZernike2} follows by making use of the quadratic transformation
   $$
  \left(1-2\xi\right)^{-2a}{_2F_1}\left( \begin{array}{c} a , a + \frac{1}2   \\ a + b + \frac{1}2 \end{array}\bigg | \frac{4\xi(\xi-1)}{(2\xi-1)^2}\right)
  =  {_2F_1}\left( \begin{array}{c} 2a , 2b   \\ a + b + \frac{1}2 \end{array}\bigg | \xi \right); \quad |\xi| < 1,
   $$
   (see \cite[p. 176]{AndrewsAskeyRoy99} or also \cite[p. 180]{WangGuo89})
    with $2a=\gamma+1$ and $2b=\gamma$, combined with the fact that
     \begin{align}
     {_2F_1}\left( \begin{array}{c} \alpha , \beta \\ \alpha \end{array}\bigg | x\right) = {_1F_0}\left( \begin{array}{c}  \beta \\ - \end{array}\bigg | x\right) = (1-x)^{-\beta}; \quad |x| < 1.
     \label{21-10}
  \end{align}

 The last identity \eqref{GenZernike4} can be easily checked  making use of \eqref{GenZernike1}.
\end{proof}

\begin{remark}
\eqref{GenZernike4} can also be obtained by taking $u=(1 -  |z|^2)$ and $v=\bar z$ in \eqref{GenZernike2} and next making use of the quadratic transformation
(\cite[Theorem 3.1.1, p. 125]{AndrewsAskeyRoy99})
and the fact \eqref{21-10}.
\end{remark}

Further fascinating generating functions, including those involving the product of disk polynomials, may be obtained
from their analogs for the Jacobi polynomials.
For example, from the Bailey's bilinear generating function for the Jacobi polynomials \cite{Bailey38} (see also \cite[p.3]{ChenSrivastava95}),
 one deduces the following generating function for the product of disk polynomials with equal arguments and indices but with different variables:
\begin{align*}
&\sum_{n=0}^{\infty} \frac{t^{n}\mathcal{Z}_{n,n}^{\gamma-1}(z,\bar{z})\mathcal{Z}_{n,n}^{\gamma-1}(w,\bar{w}) }{[n!(\gamma+n)_{n}]^{2}}
\\&
= \frac1{(1+t)^{\gamma}}
F_{4}\left(\frac{\gamma}{2},\frac{\gamma+1}{2};\gamma-1,\gamma;\frac{4|z|^{2}|w|^{2}t}{(1+t)^{2}},\frac{4(1-|z|^{2})(1-|w|^{2})t}{(1+t)^{2}}\right),
\end{align*}
where $F_{4}(a,b;c,c';x,y)$
stands for the fourth Appell's function (\cite[p.265]{Rainville71}). Furthermore, we have
\begin{align*}
\sum_{n=0}^{\infty}\displaystyle\frac{(-1)^{n}(2m+1)_{n}}{(2m+n+1)_{n}}\dfrac{\mathcal{Z}_{m+n,n}^{m}(z,\bar{z})}{(m+n)! n!}
=\displaystyle \dfrac{(-1)^{m}}{m!}\bar{z}^{m} \left[1-2(1-2|z|^{2})w+w^{2}\right]^{-(m+\frac{1}{2})}
 \end{align*}
which can be handled by making use of the generating relation (\cite[Eq.3, p.276]{Rainville71}).

\section{Summation formulae}
In addition to the summation formulae in Corollary \ref{cor7.2} deduced from the above generating functions,
we prove in this section further summation formulae for the disk polynomials $ \displaystyle \mathcal{Z}_{m,n}^{\gamma}(z,\bar{z})$. We begin with the following

\begin{theorem}\label{ThmSum}
 Let ${_1F_1}$ denotes the confluent hypergeometric function. We have
\begin{align}\label{Zconfluent}
    \sum_{k=0}^\infty  \frac{ (-1)^k\overline{z}^k}{k!} \frac{\mathcal{Z}_{m,k}^{\gamma}(z,\bar{z}) }{(\gamma+m+1)_k}
 = (\gamma+1)_m  \overline{z}^{m}  e^{-|z|^2}  {_1F_1} \left( \begin{array}{c} -m\\ \gamma+1 \end{array}\bigg | |z|^2-1\right) .
\end{align}
\end{theorem}
The proof relies on the following

\begin{proposition}
For every fixed positive integers $ m$ and $r$, we have
 \begin{align}
   \sum_{k=0}^\infty  \frac{ (-1)^k\overline{z}^k}{k!}   \frac{\mathcal{Z}_{m,n+r+k}^{\gamma}(z,\bar{z})}{(\gamma+m+n+1)_{r+k}}
    =  e^{-|z|^2} &\sum_{j=0}^m  \frac{(-1)^j(-m)_j (1 -  |z|^2)^{j}}{j!}   \label{GenZHa}\\
    &\times  \mathcal{Z}_{m-j,n}^{\gamma+j}(z,\bar{z}) H_{r,j}(z,\bar{z}). \nonumber
  \end{align}
\end{proposition}

\begin{proof}
Taking $ \displaystyle f(z)= z^r e^{-|z|^2}$ in the operational formula \eqref{Burch-OpForm1}  yields
\begin{align}
  \mathcal{A}_{m,n}^{\gamma}(z^r e^{-|z|^2})
 &=     \sum_{j=0}^m \frac{m!}{(m-j)!} \frac{(-1)^{j}(1 -  |z|^2)^{j}}{j!}   \mathcal{Z}_{m-j,n}^{\gamma+j}(z,\bar{z})
  \frac{\partial^j}{\partial z^j} \left( z^re^{-|z|^2}\right). \label{GenZH1}
\end{align}
The $j^{\mbox{th}}$ derivative of $ \displaystyle  z^re^{-|z|^2}$ in the right hand side of the last equality \eqref{GenZH1} is connected to the complex Hermite polynomials by
$$\frac{\partial^j}{\partial z^j} \left( z^re^{-|z|^2}\right) = (-1)^{j} e^{-|z|^2}  H_{r,j}(z,\bar{z}),$$
so that \eqref{GenZH1} is further reduced to
\begin{align}
  \mathcal{A}_{m,n}^{\gamma}(z^r e^{-|z|^2})
&=   e^{-|z|^2}   \sum_{j=0}^m \frac{m!}{(m-j)!} \frac{(1 -  |z|^2)^{j}}{j!}   \mathcal{Z}_{m-j,n}^{\gamma+j}(z,\bar{z}) H_{r,j}(z,\bar{z}).
   \label{GenZH2}
\end{align}
On the other hand, by  expanding  $ e^{-|z|^2}$ as a power series and inserting it in the expression of
 $ \displaystyle  \mathcal{A}_{m,n}^{\gamma}(z^r e^{-|z|^2})$ given through  \eqref{diff-opAmnf}, we get
\begin{align}
  \mathcal{A}_{m,n}^{\gamma}(z^r e^{-|z|^2})
&=  \sum_{k=0}^\infty  (-1)^m C^\gamma_{m,n} (1 -  |z|^2)^{-\gamma}  \frac{\partial^m}{\partial z^m} \left( z^{n+r+k} (1 -  |z|^2)^{\gamma+m}\right)
  \frac{(-1)^k \overline{z}^k}{k!}   \nonumber \\
&=  \sum_{k=0}^\infty  \left(\frac{C^\gamma_{m,n}}{C^\gamma_{m,n+r+k}}\right) \frac{(-1)^k \overline{z}^k}{k!} \mathcal{Z}_{m,n+r+k}^{\gamma}(z,\bar{z}) \nonumber \\
   &=  \sum_{k=0}^\infty  \frac{(-1)^k}{C^\gamma_{m+n,r+k}} \frac{ \overline{z}^k}{k!} \mathcal{Z}_{m,n+r+k}^{\gamma}(z,\bar{z}) .
   \label{GenZH3}
\end{align}
Equating the two right hand sides of \eqref{GenZH2} and \eqref{GenZH3} gives
\begin{align*}
    \sum_{k=0}^\infty  \frac{(-1)^k}{C^\gamma_{m+n,r+k}} \frac{ \overline{z}^k}{k!} \mathcal{Z}_{m,n+r+k}^{\gamma}(z,\bar{z})
    =  e^{-|z|^2} \sum_{j=0}^m   \frac{m!}{(m-j)!} \frac{(1 -  |z|^2)^{j}}{j!}   \mathcal{Z}_{m-j,n}^{\gamma+j}(z,\bar{z}) H_{r,j}(z,\bar{z}),
\end{align*}
which yields \eqref{GenZHa}.
\end{proof}

\begin{proof}[Proof of Theorem \ref{ThmSum}]
This is in fact a particular case of \eqref{GenZHa}. Indeed, by taking $r=0$ and $n=0$ and keeping in mind that
$H_{0,j}(z,\bar{z})=  \overline{z}^j$ and $ \displaystyle  \mathcal{Z}_{m-j,0}^{\gamma+j}(z,\bar{z}) = (\gamma+j+1)_{m-j}\overline{z}^{m-j}$, we get \eqref{Zconfluent}.
\end{proof}

The following result shows that the monomial $z^m$ restricted to the unit disk has an expansion in terms of the polynomials
$\mathcal{Z}_{j,k}^{\alpha}(z,\bar{z})$.

\begin{theorem}\label{thm-monomial}
For every fixed positive integer $ m$, we have
 \begin{align} \label{monomial}
  \overline{z}^{m}
 = m! e^{-|z|^2} \sum_{n=0}^\infty \sum_{j=0}^m \frac{(-1)^j\overline{z}^{n+j} (1 -  |z|^2)^{j}}{j! (\gamma+1)_{m+n}}
  \frac{ \mathcal{Z}_{m-j,n}^{\gamma+j}(z,\bar{z})}{(m-j)!n! }.
\end{align}
\end{theorem}

\begin{proof}
The action of the operator  $\mathcal{A}_{m,n}^{\gamma}$ in \eqref{diff-opAmnf} on $f(z)=  \overline{z}^n e^{-|z|^2}$ is given by
\begin{align*}
  \mathcal{A}_{m,n}^{\gamma}(\overline{z}^n e^{-|z|^2}) =  (-1)^m C^\gamma_{m,n} (1 -  |z|^2)^{-\gamma}  \frac{\partial^m}{\partial z^m} \left( |z|^{2n} (1 -  |z|^2)^{\gamma+m}e^{-|z|^2}\right)
\end{align*}
so that
\begin{align}
\sum_{n=0}^\infty \frac{ \mathcal{A}_{m,n}^{\gamma}(\overline{z}^n e^{-|z|^2})}{n! C^\gamma_{m,n}}
 = (-1)^m   (1 -  |z|^2)^{-\gamma}  \frac{\partial^m}{\partial z^m} \left( (1 -  |z|^2)^{\gamma+m}\right)
 =  (\gamma+1)_m \overline{z}^{m}. \label{GenFOpF1}
\end{align}
The last equality follows since $ \displaystyle  \frac{\partial^{k}}{\partial z^k}((1 -  |z|^2)^{\beta}) = (-\beta)_k {\bar z}^k (1 -  |z|^2)^{\beta-k}$ and $(-a)_k=(-1)^k(a-k+1)_k$. On the other hand, using the operational formula \eqref{Burch-OpForm1}, the left
 hand side of \eqref{GenFOpF1} can be rewritten as
 \begin{align}\label{GenFOpF2}
\sum_{n=0}^\infty \frac{ \mathcal{A}_{m,n}^{\gamma}(\overline{z}^n e^{-|z|^2})}{n! C^\gamma_{m,n}}
  = m!e^{-|z|^2}\sum_{n=0}^\infty \sum_{j=0}^m \frac{(-1)^j\overline{z}^{n+j} (1 -  |z|^2)^{j}}{j! C^\gamma_{m,n}}
  \frac{ \mathcal{Z}_{m-j,n}^{\gamma+j}(z,\bar{z})}{(m-j)!n! }\,,
\end{align}
and so \eqref{monomial} follows by equating the right hand sides of \eqref{GenFOpF1} and \eqref{GenFOpF2}.
\end{proof}

We conclude the paper with the following summation formula involving the $ \displaystyle \mathcal{Z}_{m,n}^{\gamma}(z,\bar{z})$ and that follows from Theorem \ref{thm-monomial} above. Namely,  we assert
\begin{corollary} We have
 \begin{align} \label{GenFctMon-e}
e^{\overline{z}(1+z)}
 = \sum_{m=0}^\infty \sum_{n=0}^\infty \sum_{j=0}^\infty \frac{(-1)^j\overline{z}^{n+j} (1 -  |z|^2)^{j}}{j! (\gamma+1)_{m+j+n}}
  \frac{ \mathcal{Z}_{m,n}^{\gamma+j}(z,\bar{z})}{m!n! }.
\end{align}
\end{corollary}

\begin{proof}
This follows from \eqref{monomial}  since
$ \sum\limits_{m=0}^\infty  \sum\limits_{j=0}^m A(m,j) = \sum\limits_{m=0}^\infty \sum\limits_{j=0}^\infty A(m+j,j)$.
\end{proof}

\noindent{\bf Acknowledegement:}
The authors are indebted to the anonymous referee for providing insightful comments, remarks and suggestions.
We would like to thank Prof. Abdellah Sebbar for the effort to spend time to read the paper and make useful corrections.
The assistance of the members of the seminars ``Partial differential equations and spectral geometry" is gratefully acknowledged.

\end{document}